\documentclass[twoside,10pt,reqno]{amsart}
\usepackage{amssymb}

\usepackage[utf8]{inputenc}
\usepackage[T1]{fontenc}

\usepackage{enumerate}

\let\comp=\circ
\let\wtilde=\widetilde
\let\what=\widehat

\newtheorem{theorem}{Theorem}
\newtheorem{lemma}[theorem]{Lemma}
\newtheorem{corollary}[theorem]{Corollary}
\newtheorem{fact}[theorem]{Fact}
\theoremstyle{definition}
\newtheorem{definition}[theorem]{Definition}
\newtheorem{problem}[theorem]{Problem}

\newcommand\setsep{;\allowbreak\ } 

\newcommand\absb[2]{\csname#1l\endcsname|#2\csname#1r\endcsname|}
\newcommand\norm[1]{\mathopen\|#1\mathclose\|}

\newcommand\normb[2]{\csname#1l\endcsname\|#2\csname#1r\endcsname\|}

\newcommand\al{\alpha}
\newcommand\de{\delta}
\newcommand\ve{\varepsilon}
\newcommand\ga{\gamma}
\newcommand\om{\omega}
\newcommand\Ga{\Gamma}
\newcommand\Om{\Omega}
\newcommand\sg{\sigma}

\newcommand\N{{\mathbb N}}
\newcommand\R{{\mathbb R}}
\newcommand\cntum{{\mathfrak c}}

\newcommand\restr[1]{\mathclose\restriction_{#1}}

\newcommand\mc{\mathcal}

\DeclareMathOperator{\diam}{diam}
\DeclareMathOperator{\card}{card}
\DeclareMathOperator{\dens}{dens}
\DeclareMathOperator{\pc}{pc}
\DeclareMathOperator{\rpc}{rpc}
\DeclareMathOperator{\ord}{ord}
\DeclareMathOperator{\cf}{cf}
\DeclareMathOperator{\cspan}{\overline{span}}
\DeclareMathOperator{\spn}{span}
\DeclareMathOperator{\supp}{supp}

\begin{document}
\title[Point character of Banach spaces  and non-linear embeddings into~$c_0(\Ga)$]{Remarks on the point character of Banach spaces and non-linear embeddings into~$c_0(\Ga)$}
\author{Petr Hájek}
\address{Department of Mathematics\\Faculty of Electrical Engineering, Czech Technical University in Prague\\Technická 2\\166~27 Praha~6\\Czech Republic}
\email{hajekpe8@fel.cvut.cz}
\author{Michal Johanis}
\address{Charles University, Faculty of Mathematics and Physics\\Department of Mathematical Analysis\\Sokolovská~83\\186~75 Praha~8\\Czech Republic}
\email{johanis@karlin.mff.cuni.cz}
\author{Thomas Schlumprecht}
\address{Department of Mathematics, Texas A\&M University\\College Station, TX 77843, USA\\and Department of Mathematics\\Faculty of Electrical Engineering, Czech Technical University in Prague\\Technická 2\\166~27 Praha~6\\Czech Republic}
\email{t-schlumprecht@tamu.edu}
\date{November 2023}
\thanks{The work was supported in part by GA23-04776S, by grant SGS22/053/OHK3/1T/13 of CTU in Prague, and the grant DMS-2054443 from the National Science Foundation.}
\dedicatory{This work is dedicated to Gilles Godefroy.}
\begin{abstract}
We give a brief survey of the results on coarse or uniform embeddings of Banach spaces into $c_0(\Ga)$ and the point character of Banach spaces.
In the process we prove several new results in this direction (for example we determine the point character of the spaces $L_p(\mu)$, $1\le p\le2$)
solving open problems posed by C.~Avart, P.~Komjáth, and V.~Rödl and by G.~Godefroy, G.~Lancien, and V.~Zizler.
In particular, we show that $X=L_p(\mu)$, $1\le p<\infty$, bi-Lipschitz embeds into $c_0(\Ga)$ if and only if $\dens X<\om_\om$.
\end{abstract}

\maketitle

The Banach space $c_0$ plays a fundamental role both in the linear and non-linear structural theory of Banach spaces.
Indeed, by the famous result of Mordecay Zippin \cite{Zip} it is the unique separable Banach space which is linearly complemented in every separable superspace.
On the other hand, a celebrated result in non-linear functional analysis, due to Israel Aharoni (Theorem~\ref{t:embed-c0} below), claims that every separable metric space admits a bi-Lipschitz embedding into the Banach space $c_0$.
It is apparently unknown if the latter property also admits a converse statement, namely if a Banach space that contains a bi-Lipschitz copy of $c_0$ also contains a linear copy thereof.

Our present note is focusing on embeddings into the non-separable version of $c_0$, namely the space $c_0(\Ga)$.
We are going to survey some known results and prove several new theorems.
The main body of work in this area is due to Jan Pelant and Vojtěch Rödl, and their coauthors.
These researchers were originally motivated by studying the covering properties of metric (or uniform) spaces,
and the connection with the uniform embeddings into $c_0(\Ga)$ was discovered somewhat later by J.~Pelant (Theorem~\ref{t:rpc-embed} below).
In fact, for normed linear spaces the existence of uniform embeddings is equivalent to the existence of bi-Lipschitz ones (this was known to J.~Pelant) as well as the coarse ones (a result of Andrew Swift).
One of our main new contributions in this note is contained in Theorem~\ref{t:rpc-ball}, which implies in particular that the embeddability of a normed linear space into $c_0(\Ga)$ is equivalent to the embeddability of its unit ball.
Theorem~\ref{t:embedc0-char} summarises known equivalent conditions for a normed linear space to be uniformly embeddable into $c_0(\Ga)$.
In Theorem~\ref{t:cotype-rpc-lim}, resp. Corollary~\ref{c:cotype-no_emb} we improve the results of \cite{PR92} and \cite{HS18}, showing that density $\om_\om$ and larger is an obstacle for embeddings into $c_0(\Ga)$ for all normed linear spaces of a non-trivial cotype.

The rest of the paper depends to a large extent on the characterisation of those sets $\Lambda$, for which $\ell_1(\Lambda)$ embeds into some $c_0(\Ga)$, which was obtained in \cite{PR92} and \cite{AKR13} (Theorems~\ref{t:rpc-lp-lim} and~\ref{t:rpc-l1}).
Namely, such embeddings exist if and only if $\card\Lambda<\omega_\omega$.
Using the uniform equivalence of unit balls of certain Banach lattices (Theorems~\ref{t:cotype_basis-embed_l1} and \ref{t:cotype_lat-embed_l1} below; in particular spaces of non-trivial cotype with an unconditional basis)
together with the fundamental embeddability result and Theorem~\ref{t:embedc0-char} we prove in Corollary~\ref{c:Lp,lat-embed} that several classes of Banach spaces of a non-trivial cotype and density less than $\om_\om$ embed into $c_0(\Ga)$.
The result holds true in particular for spaces $L_p(\mu)$, $1\le p<\infty$.
Finally, in Corollary~\ref{c:Lip-ap} we apply our results to the theory of approximations of Lipschitz mappings by smooth Lipschitz mapping in case when the range space is an absolute Lipschitz retract (e.g. any separable $C(K)$-space).

Let us now pass to the technical part of our note.
Let $X$ be a metric space.
By $U(x,r)$, resp. $B(x,r)$, resp. $S(x,r)$ we denote the open ball, resp. closed ball, resp. sphere centred at~$x\in X$ with radius~$r>0$.
If it is necessary to distinguish the metric space in which the ball is taken we use the notation $U_X(x,r)$ etc.
We begin by introducing several important concepts regarding uniformity properties of non-linear mappings between metric spaces and the related non-linear embeddings.

\begin{definition}
Let $(X,\rho)$ and $(Y,\sg)$ be metric spaces and $f\colon X\to Y$.
We say that $f$ is of
\begin{itemize}
\item bounded expansion if for every $d>0$ there exists $K\ge0$ such that $\sg\bigl(f(x),f(y)\bigr)\le K$ whenever $x,y\in X$ are such that $\rho(x,y)\le d$;
\item bounded compression if for every $K>0$ there exists $d>0$ such that $\sg\bigl(f(x),f(y)\bigr)\ge K$ whenever $x,y\in X$ are such that $\rho(x,y)\ge d$.
\end{itemize}
The mapping $f$ is called a \emph{coarse embedding} if it is both of bounded expansion and bounded compression.
It is called a \emph{uniform embedding} if it is one-to-one and both $f$ and $f^{-1}$ are uniformly continuous.
It is called a \emph{bi-Lipschitz embedding} if it is one-to-one and both $f$ and $f^{-1}$ are Lipschitz.
\end{definition}

Note that a coarse embedding may not be an embedding (i.e. one-to-one) and if $f$ is one-to-one, then it is a coarse embedding if and only if both $f$ and $f^{-1}$ are of bounded expansion.
Further, the notion of coarse embedding is non-trivial only for unbounded metric spaces.

The theory of Lipschitz mappings between separable Banach spaces, which is based on differentiability concepts, has been extensively developed by many authors (see \cite{BenLin}).
In particular, if $X, Y$ are separable Banach spaces and $Y$ has the RNP,
then the existence of a bi-Lipschitz embedding of $X$ into $Y$ implies that $X$ is linearly isomorphic to a subspace of $Y$, \cite{HM82}.
However, the space $c_0$ fails the RNP, so this theory cannot be applied in the situation of embeddings into $c_0$ (or $c_0(\Ga)$).
Instead, we have a fundamental result due to I.~Aharoni:

\begin{theorem}\cite{A77}\label{t:embed-c0}
Every separable metric space admits a bi-Lipschitz embedding into $c_0^+$ (the non-negative cone of $c_0$).
\end{theorem}

J.~Pelant \cite{P94} has found the optimal bi-Lipschitz constant in the above result to be $3$ (see also \cite{KL08}).

Our focus will be on the non-separable version of the Aharoni theorem, i.e. embeddings (uniform, coarse, or bi-Lipschitz) of Banach spaces into $c_0(\Ga)$.
This problem, in the more general setting of embedding metric, or even uniform spaces, has an interesting history.
In our note we will restrict our attention (i.e. we will introduce and treat the relevant concepts) to the metric case.

We will be using topological approach introduced by J.~Pelant, which appears unrelated at a first glance.
A covering of a set $A$ is a collection $\mc U$ of subsets of $A$ such that $A=\bigcup_{U\in\mc U}U$.
A covering $\mc U$ is called point-finite if for every $x\in A$ the set $\{U\in\mc U\setsep x\in U\}$ is finite.
A covering $\mc V$ of $A$ is called a refinement of a covering $\mc U$ if for every $V\in\mc V$ there exists a $U\in\mc U$ such that $V\subset U$.
By the well-known Stone theorem \cite{St48} every metric space $X$ is paracompact, i.e. every open covering of $X$ has a locally finite open refinement
(in particular every open covering has a point-finite refinement).

A uniform version of the latter property can be formulated as follows.
Let $X$ be a metric space.
By $\mc U(r)=\{U(x,r)\setsep x\in X\}$ we denote the full uniform covering of $X$.
A covering $\mc U$ of $X$ is called uniform if there exists an $r>0$ such $\mc U(r)$ refines $\mc U$ (in such case we say that $\mc U$ is $r$-uniform).
It is called uniformly bounded if there is $R>0$ such that $\diam U\le R$ for every $U\in\mc U$ (in such case we say that $\mc U$ is $R$-bounded).

\begin{definition}[\cite{PHK06}]
A metric space $X$ is said to have the {uniform Stone property} if every uniform covering of $X$ has a point-finite uniform refinement.
\end{definition}

\begin{definition}[\cite{S18}]
A metric space $X$ is said to have the {coarse Stone property} if every uniformly bounded covering of $X$ is a refinement of a point-finite uniformly bounded covering of~$X$.
\end{definition}

Arthur H. Stone has asked (see \cite{P77}) if a uniform version of the Stone theorem holds,
i.e. if each uniform covering of a metric space has a locally finite uniform refinement (or, equivalently, a point-finite uniform refinement, see e.g. \cite{I64}).
In fact, the problem was posed in the setting of uniform spaces which is more general; we will not introduce the concept of uniform spaces in our note.
The problem was solved in the negative independently by Evgeniy V. Shchepin \cite{S75} and J.~Pelant \cite{P75}.
E.~Shchepin proved that the Banach space $\ell_\infty(\Ga)$ fails the uniform Stone property whenever $\card\Ga>\exp_\om(\om)$.
The proof in \cite{P75} implies, in the so-called Baumgartner model of ZFC,
that for any cardinal $\kappa\ge2^{\om_1}$ the Banach space $\ell_\infty(\kappa)$ does not have the uniform Stone property.
These results have prompted the definition of the point character of a metric/uniform space (\cite{P75}, \cite{P77}, \cite{R87}, \cite{PR92}, \cite{P94}):

\begin{definition}
For a collection $\mc E$ of subsets of some set its order is defined by
\begin{itemize}
\item $\ord\mc E=n$ if $n=\max\bigl\{\card\mc D\setsep\mc D\subset\mc E,\bigcap\mc D\neq\emptyset\bigr\}\in\N$,
\item $\ord\mc E=\al$ if $\al=\sup\left\{(\card\mc D)^+\setsep\mc D\subset\mc E,\bigcap\mc D\neq\emptyset\right\}$ is an infinite cardinal.
\end{itemize}
\end{definition}

\begin{definition}
Let $X$ be a metric space.
Its point character $\pc X$ is the least cardinal $\al$ such that every uniform covering of~$X$ has a uniform refinement $\mc V$ with $\ord\mc V<\al$.
\end{definition}

It is important to note that the definition of the point character in the literature varies and some authors (usually those interested in combinatorics rather than topology)
use the following definition (which we will call the reduced point character):
\begin{definition}
Let $X$ be a metric space.
Its reduced point character $\rpc X$ is the least cardinal $\al$ such that every uniform covering of $X$ has a uniform refinement $\mc V$ with $\ord\mc V\le\al$.
\end{definition}

The difference between $\pc$ and $\rpc$ is that the cardinal $\pc$ can differentiate between certain situations occurring at limit cardinals.
For example assume that $X$ is such that $\rpc X=\om$.
Then either for each uniform covering of $X$ there is a uniform refinement with a finite order, but for different coverings this order needs to be arbitrarily high.
In this case $\pc X=\om$.
Or there is a uniform covering of $X$ that has no uniform refinement of a finite order (but has a uniform refinement of order $\om$; note that this refinement is still point-finite).
In this case $\pc X=\om_1$.
On the other hand if $\rpc X$ is a successor cardinal, then $\pc X=(\rpc X)^+$.

Note however that if $X$ is a normed linear space, then thanks to homogeneity the former situation described above cannot happen and so always $\pc X=(\rpc X)^+$.
Since we are primarily interested in normed linear spaces rather than general metric spaces, we will use the cardinal $\rpc$, which leads to more canonical formulas.

It is clear that a metric space $X$ has the uniform Stone property if and only if $\rpc X\le\om$ (if and only if $\pc X\le\om_1$).
The following crucial result of J. Pelant provides a link between the embeddings into $c_0(\Ga)$ and the uniform covering properties.
\begin{theorem}[{\cite[Corollary~2.4]{P94}}]\label{t:rpc-embed}
A metric space $X$ satisfies $\rpc X\le\om$ if and only if it admits a uniform embedding into $c_0(\Ga)$, where $\card\Ga\le\dens X$.
\end{theorem}
In general the embedding cannot be bi-Lipschitz, \cite[Corollary~4.4]{P94}.

We begin our investigation by observing several simple facts concerning the point character.
It is useful to notice that since the notion of being a refinement is transitive and in normed linear spaces we can use homogeneity, it is possible to reformulate the definitions of point character in the following way:
\begin{fact}
Let $X$ be a metric space.
\begin{itemize}
\item $\pc X$ is the least cardinal $\al$ such that for every $r>0$ the covering $\mc U(r)$ of $X$ has a uniform refinement $\mc V$ with $\ord\mc V<\al$.
\item $\pc X$ is the least cardinal $\al$ such that for every $r>0$ there exists an $r$-bounded uniform covering $\mc U$ of $X$ with $\ord\mc U<\al$.
\end{itemize}
If $X$ is a normed linear space, then
\begin{itemize}
\item $\pc X$ is the least cardinal $\al$ such that for some $r>0$ the covering $\mc U(r)$ of $X$ has a uniform refinement $\mc V$ with $\ord\mc V<\al$.
\item $\pc X$ is the least cardinal $\al$ such that there exists a bounded uniform covering $\mc U$ of $X$ with $\ord\mc U<\al$.
\end{itemize}

Analogous statements hold for the cardinal $\rpc$.
\end{fact}

We will use the above reformulations freely without mention.

\begin{fact}\label{f:rpc_embedded}
Let $X$, $Y$ be metric spaces.
\begin{enumerate}[(a)]
\item If $Y$ is uniformly homeomorphic to $X$, then $\pc Y=\pc X$ and $\rpc Y=\rpc X$.
\item If $Y\subset X$, then $\pc Y\le\pc X$ and $\rpc Y\le\rpc X$.
\end{enumerate}
\end{fact}
\begin{proof}
(a)
Let $\Phi\colon (Y,\sg)\to(X,\rho)$ be a uniform homeomorphism.
Let $r>0$.
There is $\de>0$ such that $\sg\bigl(\Phi^{-1}(x),\Phi^{-1}(y)\bigr)\le r$ whenever $x,y\in X$, $\rho(x,y)\le\de$.
Let $\mc V$ be a $\de$-bounded uniform covering of $X$ with $\ord\mc V<\pc X$ (resp. $\ord\mc V\le\rpc X$).
Let $s>0$ be such that $\mc V$ is $s$-uniform.
There is $\ve>0$ such that $\rho\bigl(\Phi(x),\Phi(y)\bigr)<s$ whenever $x,y\in Y$, $\sg(x,y)<\ve$.
Then $U_Y(x,\ve)\subset\Phi^{-1}\bigl(U_X(\Phi(x),s)\bigr)$ for every $x\in Y$ and so $\mc U=\{\Phi^{-1}(V)\setsep V\in\mc V\}$ is an $r$-bounded $\ve$-uniform covering of~$Y$.
Also, $\ord\mc U=\ord\mc V$, since $\Phi$ is a bijection.
Therefore $\pc Y\le\pc X$, resp. $\rpc Y\le\rpc X$.
The reverse inequalities follow by symmetry.

(b)
Let $r>0$.
Let $\mc V$ be an $r$-bounded uniform covering of $X$ with $\ord\mc V<\pc X$, resp. $\ord\mc V\le\rpc X$.
Let $s>0$ be such that $\mc V$ is $s$-uniform.
Put $\mc U=\{V\cap Y\setsep V\in\mc V\}$.
For every $x\in Y$ there is $V\in\mc V$ such that $U_X(x,s)\subset V$ and so $U_Y(x,s)\subset V\cap Y$.
Thus $\mc U$ is an $r$-bounded $s$-uniform covering of $Y$.
Finally, $\ord\mc W\le\ord\mc V$.
Hence $\pc Y\le\pc X$, resp. $\rpc Y\le\rpc X$.
\end{proof}

Next, let us mention an easy (but loose) estimate of the cardinal $\rpc$, see e.g. \cite[Lemma~1.1]{P94}:
\begin{fact}\label{f:rpc<=dens}
Let $X$ be a metric space.
Then $\rpc X\le\dens X$.
\end{fact}

The topological theory of dimension yields the following result: $\rpc B_{\R^n}=\rpc\R^n=n+1$ (see e.g. \cite[Theorem~11]{Sm} with \cite[Theorem~V.5]{I64}).
Combining this with Fact~\ref{f:rpc_embedded} it follows that if $X$ is an infinite-dimensional normed linear space, then $\rpc X\ge\rpc B_X\ge\om$.

We will make use of the following lemma, which is based on a trick from Mary Ellen Rudin's proof of the Stone paracompactness theorem (cf. also \cite{AKR13}).
\begin{lemma}\label{l:ord-partition}
Let $\mc U$ be an $r$-uniform covering of a metric space $X$.
Assume that $\al$ is an infinite cardinal such that $\mc U=\bigcup_{\ga\in\Ga}\mc U_\ga$ with $\ord\mc U_\ga\le\al$ for each $\ga\in\Ga$ and $\card\Ga\le\cf\al$.
Then $\mc U$ has an $\frac r2$-uniform refinement $\mc V$ with $\ord\mc V\le\al$.
\end{lemma}
\begin{proof}
We may assume without loss of generality that $\Ga$ is the ordinal interval $[1,\card\Ga)$ and that $\mc U_\ga$, $\ga\in\Ga$, are pairwise disjoint
in the sense that each $U\in \mc U$ belongs to precisely one $\mc U_\ga$.
On each $\mc U_\ga$, $\ga\in\Ga$, choose some well-ordering $\le$.
Let $\preceq$ be a lexicographic ordering on $\mc U$ induced by $(\Ga,\le)$ and $(\mc U_\ga,\le)$.
Given any $A\subset X$ let us denote $\wtilde A=\{x\in X\setsep U(x,r)\subset A\}$.
For $U\in\mc U$ set $\what U=U\setminus\bigcup_{V\in\mc U, V\prec U}\wtilde V$.
Then $\mc V=\{\what U\setsep U\in \mc U\}$ is the desired refinement.
Indeed, let $x\in X$.
Let $U\subset\mc U$ be the first one in the ordering $\preceq$ such that $U(x,\frac r2)\subset U$ (such $U$ exists, since $\mc U$ is $r$-uniform).
It follows that $U(x,\frac r2)\cap\wtilde V=\emptyset$ for every $V\in\mc U$, $V\prec U$, since $U(x,\frac r2)\subset U(z,r)$ for any $z\in U(x,\frac r2)$.
Consequently, $U(x,\frac r2)\subset\what U$.

To see that $\ord\mc V\le\al$ let $\mc D\subset\mc V$ be such that $\bigcap\mc D\neq\emptyset$.
For $\ga\in\Ga$ set $\mc D_\ga=\{\what V\in\mc D\setsep V\in\mc U_\ga\}$.
Then $\mc D=\bigcup_{\ga\in\Ga}\mc D_\ga$ and $(\card\mc D_\ga)^+\le\al$, since $\bigcap\mc D_\ga\neq\emptyset$ in case that $\mc D_\ga\neq\emptyset$.
Let $x\in\bigcap\mc D$ and let $U\in\mc U$ be the first one in the ordering $\preceq$ such that $U(x,r)\subset U$.
It follows that $x\in\wtilde U$ and consequently $x\notin\what V$ for any $V\in\mc U$, $V\succ U$.
Let $\beta\in\Ga$ be such that $U\in\mc U_\beta$.
Then $\mc D_\ga=\emptyset$ for $\ga>\beta$ and so $\mc D=\bigcup_{\ga\le\beta}\mc D_\ga$.
Since $\card\mc D_\ga<\al$ and $\beta<\cf\al$, it follows that $\card\mc D<\al$.
Therefore $(\card\mc D)^+\le\al$.
\end{proof}

One of our main tools is the following result:
\begin{theorem}\label{t:rpc-ball}
Let $X$ be a normed linear space.
If $\rpc B_X\le\al$, where $\al$ is an infinite regular cardinal, then $\rpc X\le\al$.
Consequently if $\rpc B_X$ is an infinite regular cardinal, then $\rpc X=\rpc B_X$.
\end{theorem}
\begin{proof}
In this proof every covering of any $A\subset X$ is understood as a covering of the metric space $A$, i.e. a covering by subsets of $A$.
Let $r>0$ be such that $B_X$ has a $\frac12$-bounded $r$-uniform covering of order at most $\al$.
The crucial step of the proof is the following observation:
\begin{itemize}
\item[($*$)]
Suppose that $A\subset X$ has a $\frac12$-bounded $r$-uniform covering of order at most $\al$.
Then $B=2A$ also has a $\frac12$-bounded $r$-uniform covering of order at most $\al$.
\end{itemize}
To see this let $\mc U$ be a $\frac12$-bounded $r$-uniform covering of $A$ with $\ord\mc U\le\al$.
Then $\mc V=\{2U\setsep U\in\mc U\}$ is a $1$-bounded $2r$-uniform covering of $B$ with $\ord\mc V\le\al$.
Since each $V\in\mc V$ is a subset of some ball $B_X(w,1)$, it has a $\frac12$-bounded $r$-uniform covering $\mc W_V$ with $\ord\mc W_V\le\al$.
Set $\mc W=\bigcup_{V\in\mc V}\mc W_V$.
Then $\mc W$ is clearly $\frac12$-bounded.
Further, given any $x\in B$ there is $V\in\mc V$ such that $U_B(x,r)=U_X(x,r)\cap B\subset V$ and there is $W\in\mc W_V$ such that $U_V(x,r)=U_X(x,r)\cap V\subset W$.
Hence $U_B(x,r)=U_V(x,r)\subset W$.
It follows that $\mc W$ is an $r$-uniform covering of $B$.

Finally, to see that $\ord\mc W\le\al$ let $\mc D\subset\mc W$ be such that $\bigcap\mc D\neq\emptyset$.
For $V\in\mc V$ set $\mc D_V=\{W\in\mc D\setsep W\in\mc W_V\}$.
Let $x\in\bigcap\mc D$ and let $\mc A=\{V\in\mc V\setsep x\in V\}$.
Then $\bigcap\mc A\neq\emptyset$ and hence $\card\mc A<\al=\cf\al$.
Note that $W\subset V$ whenever $W\in\mc D_V$ and $V\in\mc V$.
Thus $\mc D_V=\emptyset$ for $V\in\mc V\setminus\mc A$.
Hence $\mc D=\bigcup_{V\in\mc A}\mc D_V$.
Since $\card\mc D_V<\al$ for $V\in\mc A$ as $\bigcap\mc D_V\neq\emptyset$,
it follows that $\card\mc D<\al$ and so $(\card\mc D)^+\le\al$

To finish, using observation ($*$) inductively starting with $A=B_X$ we conclude that for each $n\in\N$ the ball $B_X(0,2^n)$ has a $\frac12$-bounded $r$-uniform covering $\mc U_n$ with $\ord\mc U_n\le\al$.
It follows that $\mc U=\bigcup_{n=1}^\infty\mc U_n$ is a uniformly bounded uniform covering of~$X$:
for each $x\in X$ consider $n\in\N$ satisfying $2^n\ge\norm x+r$, then $U_X(x,r)\subset B_X(0,2^n)$ and there is $U\in\mc U_n$ such that $U_X(x,r)\subset U\subset B_X(0,2^n)$.
Finally, $\mc U$ has a uniform refinement of order at most $\al$ by Lemma~\ref{l:ord-partition}.

The last statement of the theorem follows from Fact~\ref{f:rpc_embedded}(b).
\end{proof}

The following theorem, which summarises properties equivalent to uniform embedding into $c_0(\Ga)$, is a combination of results of J.~Pelant \cite{P94}, A.~Swift \cite{S18}, as well as some new ones:
\begin{theorem}\label{t:embedc0-char}
Let $X$ be a normed linear space.
The following statements are equivalent:
\begin{enumerate}[(i)]
\item $\rpc X\le\om$.
\item $X$ has the uniform Stone property.
\item $X$ has the coarse Stone property.
\item $X$ admits a coarse embedding into $c_0(\Ga)$ for some set $\Ga$.
\item $X$ admits a uniform embedding into $c_0(\Ga)$ for some set $\Ga$.
\item $X$ admits a bi-Lipschitz embedding into $c_0(\Ga)$ for some set $\Ga$.
\item $\rpc B_X\le\om$.
\item $B_X$ admits a uniform embedding into $c_0(\Ga)$ for some set $\Ga$.
\item $B_X$ admits a bi-Lipschitz embedding into $c_0(\Ga)$ for some set $\Ga$.
\end{enumerate}
\end{theorem}
The equivalence with statements (vii)--(ix) is new.
\begin{proof}[Proof of Theorem~\ref{t:embedc0-char}]
(i)$\Leftrightarrow$(ii) is just the definition.
(i)$\Leftrightarrow$(v) and (vii)$\Leftrightarrow$(viii) follow from Theorem~\ref{t:rpc-embed}.
(vi)$\Rightarrow$(v) is trivial, (v)$\Rightarrow$(vi) is in \cite[Remark~4.6(3)]{P94} (cf. also \cite[Theorem~3]{HJ08}).
(ii)$\Leftrightarrow$(iii) is in \cite[Lemma~3.7]{S18}.
(iv)$\Leftrightarrow$(v) is in \cite[Corollary~3.14]{S18}.
(vi)$\Rightarrow$(ix)$\Rightarrow$(viii) is trivial.
(vii)$\Rightarrow$(i) follows from Theorem~\ref{t:rpc-ball}.
\end{proof}

In light of the above characterisations it is clear that the problem of uniform embedding into $c_0(\Ga)$ is not just a special result and moreover it has a general answer.
However, deciding the embeddability for concrete spaces is not so easy.
Let us describe some crucial results in this direction, which were in fact obtained before Theorem~\ref{t:rpc-embed} was discovered.

V.~Rödl \cite{R87} constructed a metric space $X$ with $\card X=\om_1$ and $\rpc X=\om_1$.

An interesting problem, posed in \cite{GLZ}, is whether $\ell_\infty$ embeds uniformly into some $c_0(\Ga)$.
In view of Theorem~\ref{t:embedc0-char} this is equivalent to asking whether $\rpc\ell_\infty\le\om$.
\cite[Theorem~5.1]{Ho06} would imply that the answer is negative.
Unfortunately the argument of A.~Hohti is not correct and the problem seems to be still open.
Similarly, the authors in \cite{PR92} mention an argument of J.~Pelant that $\rpc\ell_\infty>\om$, but the details have not been published.
The best result in this direction are due to J.~Pelant.
First in \cite{P75}, using the Baumgartner model of ZFC, it is shown that $\rpc\ell_\infty(2^{\om_1})>\om$.
In \cite[Remark on p.~160]{P77}, J.~Pelant proves the next theorem.
\begin{theorem}[\cite{P77}]
If $\beta<\al$ are cardinals, where $\beta$ is regular, then $\rpc\ell_\infty(\al)>\beta$.
\end{theorem}
In particular, $\rpc\ell_\infty(\om_1)>\om$.
This seems to be ZFC only.
Considering the model of set theory in which $\cntum>\om_\om$ and using the well-known structural result that $\ell_1(\cntum)$ is isometric to a subspace of $\ell_\infty$ together with Theorem~\ref{t:rpc-lp-lim} below,
we obtain the following observation:
\begin{theorem}
$\rpc\ell_\infty\ge\om_\om$ if and only if $\cntum>\om_\om$.
\end{theorem}
\begin{proof}
$\Rightarrow$ follows from the fact that $\dens\ell_\infty=\cntum$, Fact~\ref{f:rpc<=dens}, and the fact that $\cf\cntum>\om$ (\cite[Theorem~5.16]{Jech}).
$\Leftarrow$ follows from the fact that $\ell_1(\cntum)$ is isometric to a subspace of $\ell_\infty$ (\cite[Exercise~5.34]{FHHMZ}) and Theorem~\ref{t:rpc-lp-lim}.
\end{proof}
So the answer to the question is consistently negative, but a ZFC proof is yet to be found.

Another class of Banach spaces for which $\rpc$ has been studied are $\ell_p(\Ga)$, $1\le p<\infty$.
An early result of V.~Rödl \cite{R78} is that $\rpc\ell_1(\Ga)$ can be arbitrarily large.
In \cite{PR92} (cf. also \cite{HS18}) the main result is the following:
\begin{theorem}[\cite{PR92}]\label{t:rpc-lp-lim}
If $\al$ is a limit ordinal and $1\le p<\infty$, then $\rpc\ell_p(\om_\al)=\om_\al$.
\end{theorem}
This means of course that $\rpc\ell_p(\om_\om)=\om_\om$, so $\ell_p(\om_\om)$ does not uniformly embed into any $c_0(\Ga)$.
We will generalise this result in Theorem~\ref{t:cotype-rpc-lim} below.
The main result in \cite{AKR13} is the following:
\begin{theorem}[\cite{AKR13}]\label{t:rpc-l1}
Let $\al$ be $0$ or a limit ordinal and $n\in\N_0$.
Then $\rpc\ell_1(\om_{\al+n})=\om_\al$.
\end{theorem}
Hence for $\al=0$ and any $n\in\N_0$ we get $\rpc\ell_1(\om_n)=\om$, i.e. each of these spaces admit a bi-Lipschitz embedding into some $c_0(\Ga)$.
Let us remark that the proof of the above theorem is based on a combinatorial result \cite[Theorem~3]{AKLR}, cf. \cite[Theorem~4]{AKR13}.
There seems to be a small mistake in this result: apparently there should be $q=\lfloor\frac{k-1}l\rfloor$ instead of $q=\lceil\frac{k-1}l\rceil$ in its statement.
This means that in fact the cardinal $\varphi(k,l,\al)$ is in most cases the successor of the cardinal given in \cite[Theorem~3]{AKLR}.
Nevertheless, the proof of Theorem~\ref{t:rpc-l1} in \cite{AKR13} can be easily fixed by properly redefining the parameters that relate to \cite[Theorem~3]{AKLR}.
We will extend this result in Corollary~\ref{c:rpc-Lp,lattice} below.

\begin{theorem}\label{t:cotype-rpc-lim}
Let $X$ be a normed linear space of a non-trivial cotype with $\dens X\ge\om_\al$, where $\al$ is a limit ordinal.
Then $\rpc X\ge\rpc B_X\ge\om_\al$.
Consequently, if $\dens X=\om_\al$, then $\rpc X=\rpc B_X=\om_\al$.
\end{theorem}
For the proof we need the following combinatorial lemma from \cite{PR92}.
\begin{lemma}[{\cite[Lemma~2]{PR92}}]\label{l:root}
Let $\beta$ be an ordinal and $n\in\N$.
Let $\Ga$ be a set of cardinality at least $\om_{\beta+n-1}$, let $\mc C$ be any set, and let $f\colon[\Ga]^n\to\mc C$ be a colouring such that $f(A)\neq f(B)$ whenever $A,B\in[\Ga]^n$ are disjoint.
Then there exists $\mc H\subset[\Ga]^n$ of cardinality $\om_\beta$ such that $\card\bigcap\mc H=n-1$ and $f\restr{\mc H}$ is one-to-one.
\end{lemma}

\begin{proof}[Proof of Theorem~\ref{t:cotype-rpc-lim}]
To prove the inequalities it suffices to show that $\rpc X\ge\om_\al$.
Indeed, if $\rpc B_X<\om_\al$, then there is an infinite regular cardinal $\lambda$ such that $\rpc B_X\le\lambda<\om_\al$ and hence $\rpc X\le\lambda$ by Theorem~\ref{t:rpc-ball}, a contradiction.

To prove that $\rpc X\ge\om_\al$ it suffices to show that $\ord\mc V\ge\om_\al$ for any $1$-bounded uniform covering $\mc V$ of $X$.
So let $\mc V$ be a $1$-bounded uniform covering of $X$.
It suffices to show that for any ordinal $\beta<\al$ there is a point in $X$ contained in $\om_\beta$ sets from~$\mc V$.
So let $\beta<\al$.
Further, let $\de>0$ be such that $\mc V$ is $\de$-uniform.
Let $q\ge2$ be such that $X$ is of cotype $q$ with cotype-$q$ constant $C>0$ and set $K=K_{1,q}C$, where $K_{1,q}$ is Kahane's constant for exponents~$q$ and~$1$ (see e.g. \cite[Theorem~2.76]{HJ}).
Let $n\in\N$ be such that $\frac1K\sqrt[q]n\frac\de4>1$.

Let $\{x_\ga\}_{\ga\in\Ga}\subset S(0,\frac\de3)$ be a $\frac\de4$-separated set with $\Ga=\dens X$ (see e.g. \cite[Fact~6.65]{HJ}).
Let $\{\ve_1,\ve_2,\dotsc,\ve_{2^n}\}$ be some fixed enumeration of the set $\{-1,1\}^n$.
We define a mapping $f\colon[\Ga]^n\to\mc V^{2^n}$ in the following way:
Let $A\in[\Ga]^n$ and let $\ga_1<\ga_2<\dotsb<\ga_n$ be such that $A=\{\ga_1,\dotsc,\ga_n\}$.
Put $y_{A,i}=\sum_{j=1}^n\ve_i(j)x_{\ga_j}$, $i=1,\dotsc,2^n$.
Then we set $f(A)=(V_1,\dotsc,V_{2^n})$, where $V_i\in\mc V$ are some sets such that $U(y_{A,i},\de)\subset V_i$, $i=1,\dotsc,2^n$.

Now assume that $A,B\subset[\Ga]^n$ are disjoint such that $f(A)=f(B)$.
Let $\ga_1<\ga_2<\dotsb<\ga_n$, resp. $\zeta_1<\zeta_2<\dotsb<\zeta_n$ be such that $A=\{\ga_1,\dotsc,\ga_n\}$, resp. $B=\{\zeta_1,\dotsc,\zeta_n\}$.
Then $y_{A,i}-y_{B,i}=\sum_{j=1}^n\ve_i(j)(x_{\ga_j}-x_{\zeta_j})$, $i=1,\dotsc,2^n$, and so there is $i\in\{1,\dotsc,2^n\}$ such that
\[
\norm{y_{A,i}-y_{B,i}}\ge\frac1K\Biggl(\,\sum_{j=1}^n\norm{x_{\ga_j}-x_{\zeta_j}}^q\Biggr)^{\frac1q}\ge\frac1K\biggl(n\frac{\de^q}{4^q}\biggr)^{\frac1q}>1.
\]
Since $\mc V$ is $1$-bounded, it follows that there is no $V\in\mc V$ such that $y_{A,i}\in V$ and $y_{B,i}\in V$, which contradicts the fact that $f(A)=f(B)$.

Consequently, the assumptions of Lemma~\ref{l:root} are satisfied (note that $\beta+n-1<\al$, since $\al$ is a limit ordinal).
Thus there exists $\mc H\subset[\Ga]^n$ of cardinality $\om_\beta$ such that $\card\bigcap\mc H=n-1$ and $f\restr{\mc H}$ is one-to-one.
Denote $A=\bigcap\mc H$ and let $\ga_1<\ga_2<\dotsb<\ga_{n-1}$ be such that $A=\{\ga_1,\dotsc,\ga_{n-1}\}$.
Denote $D=\{\ga\in\Ga\setsep A\cup\{\ga\}\in\mc H\}$.
Put $\ga_0=0$ and $\ga_n=\Ga$.
Since $\card D=\om_\beta$, there is $k\in\{1,\dotsc,n\}$ such that the ordinal interval $(\ga_{k-1},\ga_k)$ contains $\om_\beta$ elements of~$D$.
Denote $E=D\cap(\ga_{k-1},\ga_k)$.
Since $f\restr{\mc H}$ is one-to-one, it follows that the set $\mc E=\{f(A\cup\{\ga\})\setsep\ga\in E\}\subset\mc V^{2^n}$ has cardinality $\om_\beta$.
Therefore there is $m\in\{1,\dotsc,2^n\}$ such that the projection of $\mc E$ to the $m$th coordinate has cardinality $\om_\beta$
(otherwise $\mc E$ would be a subset of a finite cartesian product of sets with cardinality smaller that $\om_\beta$).
By taking one element in the preimage of each point of this projection we conclude that
there is $F\subset E$ of cardinality $\om_\beta$ such that $f(A\cup\{\ga\})_m\neq f(A\cup\{\zeta\})_m$ whenever $\ga,\zeta\in F$, $\ga\neq\zeta$.
Pick any $\zeta\in F$ and set $B=A\cup\{\zeta\}$.
Since
\[
\norm{y_{B,m}-y_{A\cup\{\ga\},m}}=\norm{\ve_m(k)x_\zeta-\ve_m(k)x_\ga}\le\norm{x_\zeta}+\norm{x_\ga}<\de
\]
for any $\ga\in F$, it follows that $y_{B,m}\in f(A\cup\{\ga\})_m\subset\mc V$ for every $\ga\in F$.
All of these sets are however different and this concludes the proof of the inequality $\ord\mc V\ge\om_\al$.

Finally, if $\dens X=\om_\al$, then we may apply Fact~\ref{f:rpc<=dens}.
\end{proof}

Combining the previous theorem with Theorem~\ref{t:embedc0-char} we obtain the following corollary:
\begin{corollary}\label{c:cotype-no_emb}
If $X$ is a normed linear space of a non-trivial cotype and $\dens X\ge\om_\om$, then $X$ does not embed uniformly or coarsely into any $c_0(\Ga)$.
\end{corollary}

On the other hand, in the positive direction we have the following:
\begin{theorem}\label{t:embed_H-rpc}
Let $X$ be a metric space that embeds uniformly into a Hilbert space and let $\dens X=\om_{\al+n}$, where $\al$ is $0$ or a limit ordinal and $n\in\N_0$.
Then $\rpc X\le\om_\al$.
In particular, if $\dens X<\om_\om$, then $X$ embeds uniformly into $c_0(\Ga)$ for some set $\Ga$.
\end{theorem}
\begin{proof}
Let $\Phi\colon X\to H$ be a uniform embedding into some Hilbert space $H$.
By considering $\cspan\Phi(X)$ we may assume that $\dens H=\dens X$ and $\dim H=\infty$.
By \cite{AMM85} (see \cite[Corollary~8.11]{BenLin}) the space $X$ embeds uniformly into $S_H$.
Using the classical Mazur mapping (\cite{M29}, see \cite[Theorem~9.1]{BenLin}) it follows that the space $X$ embeds uniformly into $S_{\ell_1(\Ga)}$ with $\card\Ga=\dens H=\om_{\al+n}$.
Then $\rpc X\le\om_\al$ by Theorem~\ref{t:rpc-l1} and Fact~\ref{f:rpc_embedded}.
If $\dens X<\om_\om$, then $\rpc X\le\om$, and so $X$ embeds uniformly into $c_0(\Ga)$ by Theorem~\ref{t:rpc-embed}.
\end{proof}

We would like to stress that the most important ingredient of the proof above is the fact that we can actually embed $X$ into a sphere of $H$ (or, equivalently, a ball),
which allows us to forward it into $\ell_1(\Ga)$ via the classical Mazur mapping.
Another approach is via Theorem~\ref{t:rpc-ball}.
For this we will make use of the next two results.
The first one was shown in the separable case in~\cite{OS94}.
The proof in the non-separable case is essentially the same, as we explain below.
\begin{theorem}\label{t:cotype_basis-embed_l1}
Let $X$ be a Banach space of a non-trivial cotype with a (long) unconditional basis $\{e_\ga\}_{\ga\in\Ga}$.
Then $B_X$ is uniformly homeomorphic to $B_{\ell_1(\Ga)}$.
\end{theorem}
\begin{proof}[Sketch of the proof]
As in the first part of the proof of \cite[Theorem~2.1]{OS94} we deduce that we can assume that $X$ is uniformly convex and uniformly smooth, and that $\{e_\ga\}_{\ga\in\Ga}$ is $1$-unconditional.
It is easy to see (cf. \cite[Proposition~2.9]{OS94}) that if $F\colon S_X\to S_Y$ is a uniform homeomorphism between the spheres of two Banach spaces $X$ and $Y$,
then the homogenous extension $\tilde F\colon B_X\to B_Y$, $\tilde F(x)=\norm xF(x/\norm x)$ for $x\neq0$ and $\tilde F(0)=0$ is a uniform homeomorphism between $B_X$ and $B_Y$.
Thus it is enough to find a uniform homeomorphism between the spheres of $X$ and $\ell_1(\Ga)$.

For a finite $A\subset\Ga $ we set $X_A=\spn\{e_\ga\setsep\ga\in A\}\subset X$.
From the proof of \cite[Theorem~2.1]{OS94} we can now deduce the following in the case that $\card\Ga$ is arbitrary:
For each finite $A\subset\Ga$ there is a uniform homeomorphism $F_A\colon S_{\ell_1(A)}\to S_{X_A}$ (here we understand the space $\ell_1(A)$ to be a subspace of $\ell_1(\Gamma)$) with the following properties:
\begin{enumerate}[(i)]
\item The modulus of uniform continuity of $F_A$ and $F_A^{-1}$ depends only on the modulus of uniform convexity and the modulus of smoothness of~$X$.
\item $F_A$ is support preserving and preserves the signs of the coefficients.
\item The family $\{F_A\colon\text{$A\subset\Ga$ finite}\}$ is consistent, meaning that for $A\subset B$ and $x\in \ell_1(A)$ it follows that $F_B(x)=F_A(x)$.
\end{enumerate}
From (iii) we deduce that the mapping $F\colon S_{\ell_1(\Ga)}\cap c_{00}(\Ga)\to S_X\cap c_{00}(\Ga)$ defined by $F(x)=F_A(x)$ for $\supp x\subset A$ is well defined.
(ii) implies that $F$ is bijective and from (i) we obtain that $F$ and $F^{-1}$ are uniformly continuous.
Thus $F$ extends to a uniform homeomorphism between the spheres of $\ell_1(\Ga)$ and $X$.
\end{proof}

Fouad Chaatit proved the following result (cf. also \cite[Theorem~9.7]{BenLin}):
\begin{theorem}[{\cite[Theorem 2.1]{Ch}}]\label{t:cotype_lat-embed_l1}
Let $X$ be an infinite-dimensional Banach lattice of a non-trivial cotype with a weak unit.
Then $B_X$ is uniformly homeomorphic to $B_{\ell_1(\Ga)}$ for some set $\Ga$.
\end{theorem}

As a corollary of the above theorems we obtain an answer to a problem of Christian Avart, Péter Komjáth, and V.~Rödl in \cite{AKR13}:
\begin{corollary}\label{c:rpc-Lp,lattice}
Let $X$ be an infinite-dimensional subspace of
\begin{enumerate}[(a)]
\item $L_p(\mu)$ for some measure $\mu$ and $1\le p<\infty$, or
\item a Banach lattice of a non-trivial cotype with a (long) unconditional basis or a weak unit.
\end{enumerate}
Let $\dens X=\om_{\al+n}$, where $\al$ is $0$ or a limit ordinal and $n\in\N_0$.
Then $\om_\al=\rpc B_X\le\rpc X\le\om_{\al+1}$.

Moreover, $\rpc X=\rpc B_X=\om_\al$ in the following cases:
\begin{itemize}
\item (a) holds with $p\le2$;
\item $n=0$;
\item $\om_\al$ is a regular cardinal (in particular $\al=0$).
\end{itemize}
\end{corollary}
\begin{proof}
The case $n=0$ follows from Theorem~\ref{t:cotype-rpc-lim} when $\al>0$, resp. Fact~\ref{f:rpc<=dens} when $\al=0$.
Now assume that $n>0$.
The inequalities $\om_\al\le\rpc B_X\le\rpc X$ follow again from Theorem~\ref{t:cotype-rpc-lim}.
If (a) holds with $p\le2$, then $\rpc X\le\om_\al$ by Theorem~\ref{t:embed_H-rpc} and \cite[Example on p.~191]{BenLin}.
In the other cases $B_X$ embeds uniformly into $B_{\ell_1(\Ga)}$ for some set $\Ga$:
In the case (a) we use the classical Mazur mapping (\cite{M29}, see \cite[Theorem~9.1]{BenLin}),
in the case (b) we use Theorem~\ref{t:cotype_basis-embed_l1}, resp.~\ref{t:cotype_lat-embed_l1}.
By considering $\cspan$ of the embedded image we may assume that $\card\Ga=\dens X=\om_{\al+n}$.
Using Theorem~\ref{t:rpc-l1} and Fact~\ref{f:rpc_embedded} we obtain that $\rpc B_X\le\om_\al$.
Theorem~\ref{t:rpc-ball} now implies that $\rpc X\le\om_\al$ in case that $\om_\al$ is regular, resp. $\rpc X\le\om_{\al+1}$ if $\om_\al$ is singular.
\end{proof}

We remark that it is consistent with ZFC that none of the $\om_\al$, $\al$ a limit ordinal, is actually regular, see \cite[p.~33]{Jech}.

In combination with Theorem~\ref{t:embedc0-char} we obtain the following corollary, which contains the solution to a problem of Gilles Godefroy, Gilles Lancien, and Václav Zizler in \cite{GLZ}:
\begin{corollary}\label{c:Lp,lat-embed}
Let $X$ be a subspace of
\begin{enumerate}[(a)]
\item $L_p(\mu)$ for some measure $\mu$ and $1\le p<\infty$, or
\item a Banach lattice of a non-trivial cotype with a (long) unconditional basis or a weak unit.
\end{enumerate}
If $\dens X<\om_\om$, then $X$ admits a bi-Lipschitz embedding into some $c_0(\Ga)$.
If $\dens X\ge\om_\om$, then $X$ does not admit a coarse or uniform embedding into any $c_0(\Ga)$.
\end{corollary}

The next result follows by the same reasoning as in \cite[Theorem~4.2]{HS18}.
In view of Corollary~\ref{c:Lp,lat-embed} the density assumption is optimal.
\begin{theorem}
If $X$ is a Banach space with a (long) sub-symmetric basis and $\dens X\ge\om_\om$ which admits a coarse or uniform embedding into $c_0(\Ga)$,
then $X$ is linearly isomorphic to some $c_0(\Lambda)$.
\end{theorem}

By combining Corollary~\ref{c:Lp,lat-embed}, \cite[Theorem~7.63]{HJ}, and \cite[Theorem~9]{J23} we obtain the following corollary
(note that any Lipschitz function from a subset of a metric space can be extended to the whole space with the same Lipschitz constant):
\begin{corollary}
Let $X$ be a subspace of $L_p(\mu)$ for some measure $\mu$ and $1<p<\infty$, resp. of some super-reflexive Banach lattice with a (long) unconditional basis or a weak unit, with $\dens X<\om_\om$.
Then there is a bi-Lipschitz embedding $\Phi\colon X\to c_0(\Ga)$ such that the component functions $f_\ga\comp\Phi$ are $C^1$-smooth, where $f_\ga$ are the canonical coordinate functionals on $c_0(\Ga)$.
\end{corollary}

We do not know whether in the corollary above in the case $p\ge2$ the component functions can be of higher smoothness.
Invoking \cite[Theorems~7.79 and 7.86]{HJ} we obtain a corollary on approximation of Lipschitz mappings:
\begin{corollary}\label{c:Lip-ap}
Let $X$ be a subspace of $L_p(\mu)$ for some measure $\mu$ and $1<p<\infty$, resp. of some super-reflexive Banach lattice with a (long) unconditional basis or a weak unit, with $\dens X<\om_\om$,
and let $Y$ be a Banach space that is an absolute Lipschitz retract.
Then there is a constant $C>0$ such that for any open $\Om\subset X$, any $L$-Lipschitz mapping $f\colon\Om\to Y$, and any continuous function $\ve\colon\Om\to\R^+$
there is a $CL$-Lipschitz mapping $g\in C^1(\Om;Y)$ for which $\norm{f(x)-g(x)}<\ve(x)$ for all $x\in\Om$.
\end{corollary}

The above corollary has applications for Whitney-type extension theorems, see \cite{JZ}.

\begin{problem}
Let $X$ be a super-reflexive space, resp. a WCG space of a non-trivial cotype, with $\dens X<\om_\om$.
Does $X$ embed uniformly into $c_0(\Ga)$?
\end{problem}

Note that by the result of J.~Pelant $C([0,\om_1])$ does not embed uniformly into $c_0(\Ga)$, \cite{PHK06},
so additional assumptions on $X$ must be placed here in order to expect a positive answer to this problem.

It should be noted that the cardinal $\om_\om$ appears rather frequently in dealings with certain properties of the Banach space $c_0(\Ga)$.
For example, in \cite{GKL00} it is proved that if $X$ is a WCG space with $\dens X<\om_\om$,
then the assumption that $X$ is bi-Lipschitz isomorphic to (a subspace of) $c_0(\Ga)$ implies that it is linearly isomorphic to it.
The cardinality restriction is necessary and the result fails for $\dens X=\om_\om$, \cite{GLZ}.
Similarly, the non-separable Sobczyk theorem holds under the exact same assumptions.
More precisely, let $X=c_0(\Ga)$ be a subspace of a WCG space $Y$.
If $\card\Ga<\om_\om$, then $X$ is complemented.
For $\card\Ga=\om_\om$ one can find counterexamples.

We close our paper with another problem related to the implication (ix)$\Rightarrow$(vi) in Theorem~\ref{t:embedc0-char}:
\begin{problem}
Let $X$, $Y$ be Banach spaces such that $B_X$ admits a bi-Lipschitz embedding into $Y$.
Does $X$ admit a bi-Lipschitz embedding into $Y$?
\end{problem}

The answer is positive for $Y=c_0(\Ga)$ (Theorem~\ref{t:embedc0-char}) and also for $Y$ a Hilbert space.
Indeed, let $f\colon B_X\to Y$ be an embedding such that $f$ is $L$-Lipschitz and $f^{-1}$ is $K$-Lipschitz.
For $n\in\N$ define $f_n\colon B_X(0,n)\to Y$ by $f_n(x)=nf(\frac xn)$.
Then each $f_n$ is an embedding with $f_n$ being $L$-Lipschitz and $f_n^{-1}$ being $K$-Lipschitz.
Extend each $f_n$ to $g_n\colon X\to Y$ by setting $g_n(x)=0$ for $x\in X\setminus B_X(0,n)$.
Let $\mc U$ be a free ultrafilter on $\N$ and define $g\colon X\to(Y)_{\mc U}$ by $g(x)=\bigl(g_n(x)\bigr)_{\mc U}$.
Since $\mc U$ is free, it is not difficult to see that $g$ is a bi-Lipschitz embedding.
Since $(Y)_{\mc U}$ is isometric to a Hilbert space, \cite[Theorem~3.3]{H80}, the result follows.

Note also that the uniform version of the above problem has a negative answer:
While $B_{\ell_p}$ is uniformly homeomorphic to $B_{\ell_2}$ for every $1<p<\infty$ using the classical Mazur mapping,
$\ell_p$ does not embed uniformly into $\ell_2$ for $p>2$ (\cite[p.~194]{BenLin}).

\end{document}